\documentclass[11pt]{amsart}

\usepackage{mathrsfs,graphicx,latexsym,tikz,color,euscript}
\usepackage{amsfonts,amssymb,amsmath,amscd,amsthm}
\usepackage{epstopdf}
\usepackage{hyperref}
\usepackage{upgreek}
\usepackage{comment}
\usepackage{pdfsync}

\newcommand{\rr}{\mathbb R}

\newcommand{\zz}{\mathbb Z}

\newcommand{\A}{\mathcal A}
\newcommand{\Ar}{\mathcal{A}_{\mathbf{r}}}

\newcommand{\D}{\mathcal D}

\newcommand{\br}{\mathbf{r}}  

\newcommand{\bs}{\mathbf s}
 
\newcommand{\I}{t_1, t_2}
\newcommand{\mc}{\mathcal{C}}

\newcommand{\al}{\alpha}

\newcommand{\Gm}{\Gamma}
\newcommand{\ep}{\varepsilon}

\newcommand{\dl}{\delta} 
\newcommand{\tht}{\theta}

\newcommand{\Lmd}{\Lambda}

\newcommand{\vph}{\varphi}
\newcommand{\vd}{\dot{\vph}}

\newcommand{\ey}{\frac{1}{2}}
\newcommand{\se}{\frac{2}{3}}
\newcommand{\sy}{\frac{1}{3}} 

\newcommand{\qt}{\tilde{q}}

\newcommand{\xb}{\bar{x}}

\newcommand{\la}{\langle}
\newcommand{\ra}{\rangle}

\theoremstyle{plain}
\newtheorem{thm}{Theorem}[section]
\newtheorem{prop}{Proposition}[section]
\newtheorem{lem}{Lemma}[section]

\theoremstyle{definition}
\newtheorem{dfn}{Definition}[section]

\theoremstyle{remark}
\newtheorem{rem}{Remark}[section]

\begin{document}

\title{Scattering in the planar restricted $(N+1)$-body problem}
 
%\date{\today}

\author{Guowei Yu}
\address{Chern Institute of Mathematics and LPMC, Nankai University, Tianjin, China}
\email{yugw@nankai.edu.cn}

\thanks{This work is supported by NSFC (No. 12671224), Nankai Zhide Fundation and the Fundamental Research Funds for the Central Universities.}

\begin{abstract}
In the restricted $(N+1)$-body problem, as time goes to positive or negative infinity the massless body has four possible final motions: hyperbolic, parabolic, oscillatory and bounded. In the hyperbolic case, the massless body comes from or goes to infinity with a non-zero limiting energy and a limiting direction, and in the parabolic case, it comes from or goes to infinity with a zero limiting energy and possibly a limiting direction. 

In this paper we study scattering orbits that come from infinity and then go to infinity. It will be proven that for the planar restricted $(N+1)$-body problem with the motions of the primaries being collision-free and periodic, for any two given asymptotic directions, there exists at least one hyperbolic scattering orbit for any positive limiting energy and one parabolic scattering orbit. The orbits are either always collision-free or regularizable with one collision. 
\end{abstract}

\maketitle

%%%%%%%%%%%%%%%%%%%%%%%%%%%%%%%%%%%%%%%%%%%%%%%%%%

\section{Introduction and main results}\label{sec:intro}

The restricted $(N+1)$-body problem in $\rr^d (d \ge 2)$ studies the motion of a massless body under Newton's gravitational force of $N (\ge 2)$ positive masses, $m_i$, $i =1, \dots, N$, which will be referred to as primaries. The motion of the massless body is described by the following differential equation
\begin{equation}
\label{eq;RNBP1} \ddot{x}(t) = \nabla U(x(t), t) =- \sum_{i=1}^N \frac{m_i(x(t) - q_i(t))}{|x(t) - q_i(t)|^3}, 
\end{equation}
where each $q_i(t)$ is the trajectory of $m_i$ and $U(x, t)$ is the time-dependent potential function
\begin{equation}
\label{eq:Uzt} U(x, t) = \sum_{i=1}^N \frac{m_i}{|x- q_i(t)|}.
\end{equation}
The trajectories of the primaries are given in advance and satisfy Newton's equations  
\begin{equation}
\ddot{q}_i(t) = -\sum_{ j \ne i} \frac{m_j (q_i(t) - q_j(t))}{|q_i(t) - q_j(t)|^3}, \;\;  \forall i =1, \dots, N. \label{eq:Nbody}
\end{equation}  
Moreover throughout this paper we assume they are always collision-free and bounded
\begin{equation}
\label{eq;q-coll-free-bound} |q_i(t)| \le \text{Constant} \quad \text{ and } \quad  q_i(t) \ne q_j(t), \quad \forall 1 \le i \ne j \le N, \quad \forall t \in \rr. 
\end{equation}

The restricted $(N+1)$-body problem can be seen as a limit of the $(N+1)$-body problem, when one of the masses is much smaller than the others. It is a useful model in describing the motions of asteroids, comets and satellites in gravitational systems, and has been studied by many authors including Poincar\'e \cite{Poincare1993} and Birkhoff \cite{Birkhoff1966}. Well-known examples include the circular restricted three-body problem and the elliptic restricted three-body problem, where the motions of the two primaries are assumed to be circular and elliptic periodic solutions of the two-body problem. 

One of the main goals in dynamical systems is to understand the long time behavior of the systems. For the restricted problem, according to Chazy \cite{Chazy22}, it has four possible final behaviors given below (assuming the motions of the primaries are bounded).   
\begin{itemize}
\item \textbf{hyperbolic}: $|x(t)| \to \infty$ and $|\dot{x}(t)| \to c > 0$, as $t \to \pm \infty$;
\item \textbf{parabolic}: $|x(t)| \to \infty$ and $|\dot{x}(t)| \to  0$, as $t \to \pm \infty$;
\item \textbf{bounded}: $\limsup_{t \to \pm \infty} |x(t)| < \infty$;
\item \textbf{oscillatory}: $\limsup_{t \to \pm \infty} |x(t)|= \infty$ and $\liminf_{t  \to \pm \infty} |x(t)| < \infty$.  
\end{itemize} 
\begin{rem}
    When the motions of primaries are collision-free, a solution of the restricted problem could only experience a singularity at a finite time due to binary collisions with one of the primaries and all such singularities can be regularied by Levi-Civita \cite{LC20} or Moser \cite{Moser70} and the solution can be analytically extended through the collision, see Proposition \ref{prop;Sundman}. 
\end{rem} 

Since the work of Poincar\'e \cite{Poincare1993} and Birkhoff \cite{Birkhoff1966}, bounded motions (e.g., periodic, relative periodic and quasi-periodic) have been studied extensively by many authors using different methods, see \cite{Meyer99} and the references within. Oscillatory motions were first found by Sitnikov \cite{St60} with proofs given by Alekseev \cite{Alk72} and Moser \cite{Moser73} in a restricted spatial isosceles three-body problem known as the Sitnikov problem. Later oscillatory motions were found in the planar circular and elliptic restricted three-body problems in \cite{LS80b}, \cite{GMS16} and \cite{GMSS17}. Here we only mention results in the restricted case. 

In this paper we will focus on hyperbolic and parabolic motions. The following result was proven recently by Chen and the author in \cite{CY26}. 
\begin{thm}
\label{thm;hyp-para} Under the conditions \eqref{eq;q-coll-free-bound} and $d \ge 2$, given any initial time and configuration $(t_0, x_0 ) \in \rr \times \rr^d$ and asymptotic direction $\bs \in \mathbb{S}^{d-1}: = \{ x \in \rr^d: |x| =1 \}$, for any $h>0$, 
\begin{enumerate}
	\item[(a).] there is a hyperbolic solution $x(t)$ with $x(t_0) =x_0$, $x(t)$ collision-free, for all $t >t_0$ and 
    $$ x(t) = \sqrt{2h}t\bs + \frac{m}{2h} (\ln t) \bs + o(\ln t), \text{ as  } t \to \infty, \text{ where } m = \sum_{i=1}^N m_i;$$
    \item[(b).] there is a parabolic solution $x(t)$ with $x(t_0) =x_0$, $x(t)$ collision-free, for all $t >t_0$ and 
    $$ x(t) = \al t^{\se} \bs  + o(t^{\se}), \text{ as } t \to \infty, \text{ where } \al =\left( 9m/2 \right)^{\frac{1}{3}}.$$     
\end{enumerate}
Moreover similar results hold for $t \to -\infty$.  

\end{thm}
\begin{rem}
%This result holds for all $d \ge 2$, when the motions of the primaries are collision-free and bounded, not necessarily periodic. 

The proof of this theorem in \cite{CY26} followed an approach given in \cite{PT24}, where the existence of parabolic, hyperbolic and hyperbolic-parabolic solutions in the $N$-body problem with given initial configurations and asymptotic behaviors at infinity was proven. For parabolic and hyperbolic solutions earlier proofs can also be founded in \cite{MV09}, \cite{MV20}, \cite{LYZ21}. All these works are variational in nature. 
\end{rem}

In the above theorem, $\bs$ is the limiting direction of the massless body, when it approaches infinity in both cases. Meanwhile in the hyperbolic case $h$ is the limiting energy of the massless body, and in the parabolic case, the limiting energy is zero. Since scattering orbits are those with the massless body coming from infinity and going to infinity, we may ask the following question.  

 \textbf{Question.} \emph{Given any $h>0$ and any two $\bs^{\pm} \in \mathbb{S}^{d-1}$, is there a solution $x: \rr \to \rr^d$ in the restricted problem satisfying the following asymptotic expressions?}
\begin{equation}  \label{eq;scatter}
\begin{aligned}
x(t) &=\sqrt{2h}|t|\bs^{\pm} + \frac{m}{2h}(\ln |t|) \bs^{\pm} + o(\ln |t|), &t \to \pm \infty, \quad \text{the hyperbolic case}; \\
x(t) &=\al |t|^{\se} \bs^{\pm}  + o(|t|^{\se}), &t \to \pm \infty, \quad \text{the parabolic case}.
\end{aligned}
\end{equation}

\begin{dfn}
    We say a solution $x(t)$ is \emph{an $h$-energy hyperbolic scattering orbit} between $\bs^-$ and $\bs^+$, if it satisfies the hyperbolic case in \eqref{eq;scatter}, and \emph{a parabolic scattering orbit} between $\bs^-$ and $\bs^+$ if it satisfies the parabolic case in \eqref{eq;scatter}.
    \end{dfn}    
 
In this paper, we give a positive answer to the above question for the planar restricted problem with the motion of primaries being periodic and collision-free. 
\begin{thm}
\label{thm;bi-hyp-para} When $d=2$ and $q(t)$ is periodic, for any $h>0$ and any pair of $\bs^{\pm} \in \mathbb{S}^1$, 
\begin{enumerate}
\item[(a).] there is an $h$-energy hyperbolic scattering orbit between $\bs^-$ and $\bs^+$;
\item[(b).] there is a parabolic scattering orbit between $\bs^-$ and $\bs^+$.
\end{enumerate}
Moreover these orbits are either collision-free or regularizable with one collision. 
\end{thm}

For scattering orbits in the $N$-body problem, we refer the interested readers to the recent book by Montgomery \cite{Mont25}. So far there are only a few limited results in this case, see \cite{DMMY20} and \cite{Mk26}.

Our proof of Theorem \ref{thm;bi-hyp-para} will be variational as well. Denote the set of all Sobolev paths defined on $[\I]$ by $H^1_{\I}$. It is well known that collision-free minimizers of the action functional
\begin{equation}
\label{eq;A-restricted} \A(x; t_1, t_2) = \int_{t_1}^{t_2} \ey |\dot{x}(t)|^2 + U(x(t), t) \,dt, \quad x \in H^1_{\I}
\end{equation}
are classical solutions of \eqref{eq;RNBP1}. Since the beginning of this century variational methods have had great success in the study of the $N$-body problem starting from the proofs of the Figure-Eight solution \cite{CM00} and the Hip-Hop solution \cite{CV00}. The majority of the results are based on Marchal's lemma (see \cite{C02} and \cite{FT04}) which shows a fixed-end local minimizer in the $N$-body problem must be collision-free between the endpoints. Because of the lack of a Marchal-type lemma, there are very few results in the restricted problem obtained through variational methods. The only ones that we are aware of are \cite{Terr00}, \cite{Sh19} and \cite{PT22}. 

The extension of Marchal's lemma to the restricted problem does not follow from the standard argument for the unrestricted case. The place where it fails is regarding the isolation of collision moments, where in the unrestricted it is based on a time reparameterization argument which proves the partial energy function of the bodies is continuous along a locally minimizing path even at collisions. This works in the unrestricted case as it is time-independent. However the restricted problem is time-dependent, so it does not work here. The problem was solved very recently by Chen and the author in \cite{CY26} and the corresponding Marchal-type lemma for the restricted problem  is obtained, which guarantees fixed-end local minimizers are collision-free between the endpoints (see Theorem \ref{thm;Marchal-restricted}). It played a key role in the proof of Theorem \ref{thm;hyp-para} in \cite{CY26}, and will be important for our proof of Theorem \ref{thm;bi-hyp-para} as well.  

Compared with the results in Theorem \ref{thm;hyp-para} with fixed initial time and configuration, the main difficulty in applying variational methods to scattering orbits is the lack of coercivity. One can see this by observing that the action value is zero when the massless body stays unmoved at infinity all the time. 
A possible way to force coercivity is by imposing topological constraints, but in general minimizers under topological constraints are not collision-free or even regularizable, one can already see this from the classical result of Gordon in the Kepler problem \cite{Gor77}. 

In this paper, we overcome this difficulty in the planar case by forcing the paths of the massless body to be tied with the trajectories of any two of the primaries, and we are able to show that under such a topological constraint, the minimizing path is either always collision-free or contains one regularizable binary collision. For such a collision singularity to be regularizable, the trajectory of the massless body must satisfy the conditions given in Proposition \ref{prop;Sundman}. First it needs to collide to and eject from the collision along the same relative direction. Second the relative energy should extend continuously to the collision moment. As we already mentioned the second is not so obvious for a minimizing path in the restricted problem. 

Besides coercivity, another difficulty in studying the one-sided hyperbolic/parabolic solutions or the scattering orbits through variational methods is to control the asymptotic behavior at infinity. A nice method to solve this is by introducing a suitable renormalized action functional and a proper admissible function space. This is largely due to the works by Terracini and her collaborators, see \cite{BDFT21}, \cite{PT24} and \cite{PT22}. In particular in \cite{PT22}, Paradela and Terracini studied the restricted spatial isosceles 3-body problem with two equal masses primaries moving along a negative energy homothetic collision-ejection orbit belonging to a line and the massless body moving on the plane perpendicular to this line and passing through the center of mass. In such a setting, collision between the massless body and the primaries can be excluded by non-vanishing angular momentum. They showed the existence of chaotic motions and certain parabolic scattering orbits in this problem for small but non-vanishing angular momentum.  

We finish this introduction by recalling some results about scattering orbits in the $N$-center problem which can be seen as a simplified model of the restricted $(N+1)$-body problem with all the primaries fixed. In this case, the problem is time independent and the results are much better. A comprehensive work of high energy solutions in the planar $N$-center problem including scattering orbits can be found in \cite{Knauf92} by Klein and Knauf, and later extended to the spatial case in \cite{Knauf02} by Knauf. In these works singularities are regularizable at the beginning and the positive energy solutions are transferred to geodesics of a Jacobi-Maupertuis metric. More recently, collision-free scattering orbits are studied in \cite{BDT17}, \cite{BBD18} and \cite{BDP18} through variational methods.

%%%%%%%%%%%%%%%%%%%%%%%%%%%%%%%%%%%%%%%%%%%%%%%%%

\section{Minimizers tied with a pair of masses} \label{sec;mini-top-const}

In this section we will first recall the Marchal-type lemma (Theorem \ref{thm;Marchal-restricted}) for local minimizers of the restricted problem proven in \cite{CY26}. Then we will introduce a topological constraint on the admissible paths by tying them with the trajectories of a pair of primaries (see Definition \ref{def:tied}). By a classical result of Gordon \cite{Gor77}, one cannot expect a minimizer under such a tied condition to always be collision-free. Instead we shall prove it contains at most one collision and is regularizable. As we mentioned in the introduction, the tied condition will be used to force coercivity in the proof of our main results. In the rest of the paper $C$ and $C_i$, $i \ge 1$, will denote positive constants that may vary from line to line.

For any $p, q \in \rr^2$ and $t_1 < t_2$, let $H^1_{t_1, t_2}(p, q)$ denote all Sobolev paths on $[t_1, t_2]$ from $p$ to $q$. 
\begin{thm} [\cite{CY26}]
\label{thm;Marchal-restricted} If $x \in H^1_{t_1, t_2}$ is a fixed-end local minimizer of $\A$, i.e., $\exists \ep>0$, such that 
$$ \A(x; t_1, t_2) \le \A(y; t_1, t_2), \;\; \forall y \in H^1_{t_1, t_2}(x(t_1), x(t_2)) \text{ with }  \|y-x\|_{H^1} < \ep$$
then $x(t)$ is collision-free, for all $t \in (t_1, t_2)$, and a classical solution of \eqref{eq;RNBP1}.
\end{thm}

For any path $x \in H^1_{t_1, t_2}$, we denote the set of all collision moments of $x$ by $Z_x$
$$ Z_x:= \{t \in [t_1, t_2]: \; x(t) = q_i(t), \; \text{ for some } 1 \le i \le N  \}. $$ 

\begin{dfn}
\label{dfn;isolated-coll} We say $t_0 \in (t_1, t_2)$ is \emph{an isolated collision moment} and $x(t_0)$ is \emph{an isolated collision}, if there is a $\dl>0$, such that $Z_x \cap [t_0-\dl, t_0 +\dl] = \{t_0\}$, and $x|_{[t_0 -\dl, t_0 +\dl]}$ is \emph{a collision-ejection solution} of the restricted problem, if $x|_{[t_0 -\dl, t_0)}$ and $x|_{(t_0, t_0 +\dl]}$ are solutions of \eqref{eq;RNBP1}.
\end{dfn}

\begin{prop}
    \label{prop;Sundman} Let $x|_{[t_0-\dl, t_0+\dl]}$ be a collision-ejection solution with $x(t_0) = q_{i_0}(t_0)$. 
    \begin{enumerate}
        \item[(a).] There is an $\bs^+ \in \mathbb{S}^1$ such that when $t \to t_0^+$
   $$ \begin{aligned}
x(t) - q_{i_0}(t) & =  \left(\frac{9}{2}m_{i_0} \right)^{\frac{1}{3}}(t-t_0)^{\frac{2}{3}} \bs^+ + O\left(|t-t_0|^{\frac{4}{3}} \right); \\
\dot{x}(t) - \dot{q}_{i_0}(t) & = \se \left(\frac{9}{2}m_{i_0} \right)^{\frac{1}{3}}(t-t_0)^{-\sy} \bs^+  + O\left(|t-t_0|^{\frac{1}{3}} \right),
\end{aligned}$$
and there is an $\bs^- \in \mathbb{S}^1$, such that similar estimates hold when $t \to t_0^-$. 
    \item[(b).] $x$ is a regularizable solution, if $\bs^+ = \bs^-$ and 
     $$ \lim_{t \to t_0^-} \left(\ey |\dot{x}(t) - \dot{q}_{i_0}(t)|^2 - U(x(t), t) \right) = \lim_{t \to t_0^+} \left(\ey |\dot{x}(t) - \dot{q}_{i_0}(t)|^2 - U(x(t), t) \right). $$
     \end{enumerate} 
\end{prop}

\begin{proof}
    Property (a) consists of the classical \emph{Sundman estimates}; for a proof see \cite{Sp69}. Property (b) is also well-known and a proof can be found in \cite{Zhao23}. 
\end{proof} 

Now we will explain how to tie a path with two masses. For this in the rest of the section, let's fix an arbitrary pair of masses $m_a$ and $m_b$ with $1 \le a < b \le N$, a time interval $[t_1, t_2]$ and two configurations $p, q \in \rr^2$ satisfying
$$ p \ne q_i(t_1)  \; \text{ and } q  \ne q_i(t_2), \quad 1 \le i \le N. $$

\begin{dfn}\label{def:tied}
We say a continuous path $x: [t_1, t_2] \to \rr^2$ avoiding \(q_a(t)\) and \(q_b(t)\) is \emph{tied with $m_a$ and $m_b$}, if there is no homotopy from $x$, fixing the endpoints and passing through continuous paths always avoiding  \(q_a(t)\) and \(q_b(t)\), to a path with image contained in $\rr^2 \setminus B_R$, where $R= \max_{t \in [t_1, t_2]} \{|q_a(t)|, |q_b(t)|\}$.
\end{dfn}

For any $p, q \in \rr^2$ with $|p|,|q|>R$, denote the set of all continuous paths in $H^1_{t_1, t_2}(p, q)$ avoiding $q_a(t)$ and $q_b(t)$ and tied with $m_a$ and $m_b$ by $\hat{\Gm}^{a, b}_{t_1, t_2}(p, q)$ and 
$$ \Gm^{a, b}_{t_1, t_2}(p, q) = \{ x \in H^1_{t_1, t_2}: \; x_n \rightharpoonup x \in H^1_{t_1, t_2} \text{ weakly for some }x_n\in \hat{\Gm}^{a, b}_{t_1, t_2}(p, q) \}. 
$$
Denote the subset of all paths in $H^1_{t_1, t_2}(p, q)$ colliding with $m_a$ or $m_b$ by
$$ \mc^{a, b}_{t_1, t_2}(p, q) = \{ x \in H^1_{t_1, t_2}(p, q): x(t) = q_a(t) \text{ or } q_b(t) \; \text{ for some } t \in (t_1, t_2) \}.$$ 

\begin{prop} \label{prop;Gamma} $\mc^{a, b}_{\I}(p,q) \subset \Gm^{a,b}_{t_1, t_2}(p, q)$.
\end{prop} 

\begin{proof}
This follows directly from the above definitions. 
\end{proof}

\begin{prop} \label{prop;mini-exist} $\A$ has at least one minimizer $x \in \Gm^{a, b}_{t_1, t_2}(p, q)$.
\end{prop}

\begin{proof}
This follows from a standard argument of the direct method of calculus of variations.
\end{proof}
For the rest of the section, let $x \in \Gm^{a, b}_{t_1, t_2}(p, q)$ be a minimizer given by the above proposition.

\begin{lem}
    \label{lem;Zx} $Z_x$ contains at most one point. 
\end{lem}

\begin{proof} 
    First we claim the massless body does not collide with any primaries other than $m_a$ or $m_b$. By a contradiction argument, assume $x(t_0) = q_i(t_0)$, for some $t_0 \in (t_1, t_2)$ and $i \ne a$ and $b$. Since $q|_{[\I]}$ is collision-free, for $\dl>0$ small enough, if we make any small enough fixed-end deformation of $x$ on the time interval $[t_0-\dl, t_0+\dl]$, the new path should still belong to $\Gm^{a, b}_{t_1, t_2}(p, q)$. Hence $x|_{[t_0-\dl, t_0+\dl]}$ is a fixed-end local minimizer of $\A$ with a collision at $t_0$, which is a contradiction to Theorem \ref{thm;Marchal-restricted}. This implies that
    $$ Z_x = \{ t : x(t) = q_a(t) \text{ or } q_b(t), \; \text{ for some } t \in (t_1, t_2) \}. $$ 

    Again by a contradiction argument, let's assume $Z_x$ contains two different times $\tau_1 < \tau_2$. Then we can find a $\dl>0$ small enough, such that $t_1 < \tau_1-\dl < \tau_1 + \dl < \tau_2$. By Proposition \ref{prop;Gamma}, if we make any small enough fixed-end deformation of $x$ on the time interval $[\tau_1-\dl, \tau_1 +\dl]$, the new path must still belong to $\Gm^{a, b}_{t_1, t_2}(p, q)$. Hence $x|_{[\tau_1-\dl, \tau_1 +\dl]}$ is a local minimizer with a collision at $\tau_1$, which is a contradiction to Theorem \ref{thm;Marchal-restricted}. 
\end{proof}

\begin{lem}
    \label{lem;reg} If $Z_x$ is non-empty, then $x$ is a regularizable solution of \eqref{eq;RNBP1} containing a single collision with $m_a$ or $m_b$.  
\end{lem}

A proof of this lemma will be given in Section \ref{sec;reg}. Hence we have proved the following result.
\begin{prop}
\label{prop;mini-tied} 
 The action functional $\A$ has a minimizer $x$ in $\Gm^{a, b}_{t_1, t_2}(p, q)$, which is either a collision-free solution or a regularizable solution of \eqref{eq;RNBP1} with at most one collision with $m_a$ or $m_b$.  
\end{prop}

\section{Proof of Theorem \ref{thm;bi-hyp-para}}

Recall that by assumption $q(t)$ is periodic. We will only give the details when the period of $q(t)$ is $1$, as for other periods the result follows either from a similar argument or from a rescaling of the system which makes the period of the motion of primaries to be one.

As already mentioned in the introduction, the desired hyperbolic and parabolic scattering orbits will be found as minimizers of suitable renormalized action functional, which has been used in \cite{BDFT21}, \cite{PT22} and \cite{PT24}. One of the key observations is that, for the prescribed asymptotic behavior at infinity, we can fix a reference path $\xb \in C^2(\rr, \rr^2)$ defined as \eqref{eq;xb-bi-hyp-1} and \eqref{eq;xb-bi-para-1} with the desired asymptotic behavior for different cases.   
\begin{equation}
\label{eq;xb-bi-hyp-1} \text{the hyperbolic case} \quad \begin{cases}
\xb(t) =  \sqrt{2h}(t+1)\bs^+ + r_0 \bs^+, \; & \text{ when } t \ge 1; \\
|\xb(t)| = \sqrt{2h} + r_0, \; & \text{ when } |t| \le \ey; \\
\sqrt{2h } +r_0 \le |\xb(t)| \le 2 \sqrt{2h} + r_0, \; & \text{ when } \ey \le |t| \le 1; \\
\xb(t) =  \sqrt{2h}(1-t) \bs^- + r_0 \bs^-, \; & \text{ when } t \le -1,
\end{cases},
\end{equation}

\begin{equation}
\label{eq;xb-bi-para-1} \text{the parabolic case} \quad \begin{cases} 
\xb(t) =  \al(t+1)^{\se}\bs^+ + r_0 \bs^+, \; & \text{ when } t \ge 1; \\
|\xb(t)| = \al + r_0, \; & \text{ when } |t| \le \ey; \\
\al +r_0 \le |\xb(t)| \le 2^{\se}\al + r_0, \; & \text{ when } \ey \le |t| \le 1; \\
\xb(t) =  \al(1-t)^{\se} \bs^- + r_0 \bs^-, \; & \text{ when } t \le -1.
\end{cases} 
\end{equation}
where $\al = (9\sum_{i} m_i/2)^{\sy}$ and $ r_0 = \max\{ |q_i(t)|: \; t \in [0, 1] \;\text{ and }\; \forall 1 \le i \le N  \}$.

From now on we will only consider paths with the following expression
$$ x(t) = \xb(t) + \vph(t) $$
where $\vph$ is always chosen from the following function space 
$$ \D : = \{ \vph \in H^1_{loc}(\rr, \rr^2): |\vph(0)|^2 + \int_{\rr} |\dot{\vph}|^2 \,dt <\infty \}. $$
It is a Hilbert space with the following norm and  inner product
$$ \|\vph\|_{\D} = \sqrt{\langle \vph, \vph \rangle_{\D}};$$
$$ \langle \vph, \psi\rangle_{\D} = \la \vph(0), \psi(0)\ra + \int_{\rr} \langle \dot{\vph}(t), \dot{\psi}(t) \rangle \,dt.$$

A direct computation then gives the following proposition. 
\begin{prop}
    \label{prop;phi-sol-bi-hp} $x(t) = \xb(t) +\vph(t) \in C^2(\rr, \rr^2)$ is a solution of \eqref{eq;RNBP1}, if and only if 
\begin{equation} \label{eq;RNBP-3}
\ddot{\vph}(t) = \nabla U(\xb(t) + \vph(t), t) - \ddot{\xb}(t), \; \forall t \in \rr, 
\end{equation}
which is the Euler-Lagrange equation of the following renormalized action functional 
\begin{equation} \label{eq;Ar}
\Ar(\vph) = \int_{\rr} \ey |\dot{\vph}|^2 + U(\xb + \vph, t) - U(\xb, t) - \langle \ddot{\xb}, \vph \rangle \,dt.
\end{equation}
\end{prop}

Next we need to find a proper admissible subset of $\D$, where the renormalized functional $\Ar$ is coercive. First we need a constraint to prevent the convergence of the paths to infinity in space, and this can be done by considering paths which are tied with any two of the primaries. 

Fix an arbitrary pair of masses $m_a$ and $m_b$ for the rest of the section and define
\begin{equation}
    \label{eq:Lmd} \begin{split}
\Lmd^{a,b}:=\{\vph\in\D:\ &\exists T>1\text{ such that }
 x=\xb+\vph,\quad |x(t)|>r_0\text{ for }|t|\ge T,\\
 &x|_{[-T,T]}\in \Gm^{a,b}_{-T,T}(x)\}.
\end{split}
\end{equation}
Here and below $ \Gm^{a,b}_{-T,T}(x) = \Gm^{a, b}_{-T, T}(x(-T), x(T))$. 

While the tied condition can force coercivity in space, we also need some constraint on the time parameter as the system is time periodic. For this, we first define $T_k$ as an integer time shift by $k$ 
\begin{equation} \label{eq;Tk}
    T_k \vph (t) := \vph(t + k) + \xb(t+k) - \xb(t), \;\;  \forall k \in \zz.
\end{equation}

\begin{lem}
\label{lem;D-inv} If $\vph\in\D$, then $T_k\vph\in\D$ for every $k\in\zz$.
\end{lem}
\begin{proof}
Let $s = t+k$ and $\xb_k(s)=\xb(s-k)$. Then 
$$ T_k \vph(t) = \vph(s) + \xb(s) - \xb_k(s). $$
In the hyperbolic case, by \eqref{eq;xb-bi-hyp-1}, $\dot \xb(s)-\dot \xb_k(s)$ has compact support. In the parabolic case, by \eqref{eq;xb-bi-para-1} and the mean value
formula
$$ |\dot \xb(s)-\dot \xb_k(s)| = O(|s|^{-4/3}), \quad \text{ as } |s| \to \infty.$$ 
Hence $\frac{d T_k \vph}{dt}$ belongs to $L^2(\rr)$ in both cases, and the desired result follows. 
\end{proof}

%Since $\dot{T_k\vph}(t)=\dot\vph(t+k)+\dot b(t+k)-\dot b(t)$ and $(T_k\vph)(0)$ is finite, the result follows.

\begin{lem}
\label{lem;Ah-inv-timeshift} If $\vph\in\D$ and $\Ar(\vph)<\infty$,
then $\Ar(T_k\vph)=\Ar(\vph)$ for every $k\in\zz$.
\end{lem}
\begin{proof}
For any $\vph \in \D$, set the integrand in $\Ar(\vph)$ as
$$ F_{\vph}(s) = \ey |\dot{\vph}(s)|^2 + U(\xb(s) + \vph(s), s) - U(\xb(s), s) - \langle \ddot{\xb}(s), \vph(s) \rangle.$$
Set $s=t+k$, $\xb_k(s)=\xb(s-k)$ and $d_k(s)=\xb(s)-\xb_k(s)$. Then $T_k \vph(t)  =\vph(s) + d_k(s)$ and 
$$ \begin{aligned}
    F_{T_k \vph}(s -k) & = \ey |\dot{\vph}(s) + \dot{d}_k(s)|^2 + U(\xb(s-k)+ \vph(s) + d_k(s), s-k) \\
                        & \quad - U(\xb(s-k), s-k)  - \la \ddot{\xb}(s-k), \vph(s)+ d_k(s) \ra \\
                        & = \ey |\dot{\vph}(s) + \dot{d}_k(s)|^2 + U(\xb(s)+ \vph(s), s) \\
                        & \quad - U(\xb(s-k), s-k)  - \la \ddot{\xb}(s-k), \vph(s)+ d_k(s) \ra \\
\end{aligned}
$$
where the last equation follows from the fact that $q(s +1) = q(s)$. 

By a direct computation
\begin{equation}
     \label{eq;diff-F}\begin{split}
 F_{T_k\vph}(s-k)-F_\vph(s)
 &=\frac{d}{ds}\left(
 \langle\dot d_k(s),\vph(s)\rangle-\langle\dot\xb_k(s),d_k(s)\rangle\right)\\
 &\quad+\frac12\big(|\dot\xb(s)|^2-|\dot\xb(s-k)|^2\big)
       + \big( U(\xb(s), s)-U(\xb(s-k), s-k) \big).
\end{split} 
 \end{equation} 

By the Cauchy-Schwarz inequality, $|\vph(s)|=O(1+\sqrt{|s|})$. In the hyperbolic case, for sufficiently large $|s|$,
$$\dot d_k(s)=0, \quad  \langle\dot\xb_k(s),d_k(s)\rangle=k|\dot\xb_k|^2=2hk.$$
In the parabolic case, for sufficiently large $|s|$,
\[
 d_k(s)=O(|s|^{-1/3}),\quad
 \dot d_k(s)=O(|s|^{-4/3}),\quad
 \ddot d_k(s)=O(|s|^{-7/3}),\quad
 \dot\xb_k(s)=O(|s|^{-1/3}).
\]
Therefore in either case,
\[
 \lim_{T\to\infty}
 [\langle\dot d_k,\vph\rangle-
   \langle\dot\xb_k,d_k\rangle]_{-T}^{T}=0
\]
and the first term in \eqref{eq;diff-F} is integrable on $\rr$ with the integral being zero. 

For the second term in \eqref{eq;diff-F}, notice that it has compact support in the hyperbolic case and is $O(|s|^{-5/3})$, for large $|s|$, in the parabolic case, so it is absolutely integrable. For the third in \eqref{eq;diff-F}, a direct computation gives
$$ U(\xb(s), s)-U(\xb(s-k), s-k) = \begin{cases}
O(|s|^{-2}), \quad & \text{ in the hyperbolic case},\\
O(|s|^{-5/3}), \quad & \text{ in the parabolic case}.
\end{cases}
$$
Hence it is also absolutely integrable in both cases. Moreover by a simple exercise, we can show the integrals of the second and third terms on $\rr$ are both zero. 

As a result, $F_{T_k\vph}$ is integrable on $\rr$ and
$\Ar(T_k\vph)=\Ar(\vph)$.
\end{proof}

%By the above two lemmas, if there is a minimizing sequence of paths in $\D$, we may shift the time parameter of each path by a proper integer. The new sequence of paths is still minimizing.  

By the tied condition, for any $\vph \in \Lmd^{a,b}$ there must be a moment $t_0$, such that 
$$ |\xb(t_0) +\vph(t_0)| \le r_0. $$
Meanwhile by Lemmas \ref{lem;D-inv} and \ref{lem;Ah-inv-timeshift}, we may shift the time parameter by a proper integer to make
$$  t_0 \in [-1/2, 1/2]. $$
This allows us to define the following subset of admissible paths with normalized time 
\begin{equation} \label{eq;Lmd0}
    \Lmd^{a,b}_0:=\{\vph\in\Lmd^{a,b}:\ \exists |t_0| \le 1/2
\text{ with }|\xb(t_0)+\vph(t_0)|\le r_0\}.
\end{equation}
As a result for any $\vph\in\Lmd^{a,b}_0$  
\begin{equation}
\label{eq;vph-t0-upb} |\vph(t_0)| \le r_0 + \max_{-1/2 \le t \le 1/2} |\xb(t)| : =r_1. 
\end{equation}
For the rest of the paper, set 
\begin{equation}
\label{eq;w} w(t) = 1 + |t|, \quad t \in \rr. 
\end{equation}
By the Cauchy-Schwarz inequality, 
\begin{equation}\label{eq;D-pointwise}
 |\vph(t)|\le |\vph(0)|+\|\dot\vph\|_{L^2}\sqrt{|t|} \le C\|\vph\|_{\D}\sqrt{w(t)}.
\end{equation} 

\begin{lem}
\label{lem;Dnorm-L2norm} If $\vph_n \in \Lmd^{a, b}_0$ is a sequence satisfying $\|\vph_n\|_{\D} \to \infty$, when $n \to \infty$, then so is $\|\dot{\vph}_n\|_{L^2}$. 
\end{lem} 
\begin{proof}
By \eqref{eq;D-pointwise},  
$$ |\vph(0)| \le r_1 + \|\dot{\vph}\|_{L^2},$$
which means 
$$ \|\vph\|_{\D}^2 \le ( r_1 + \|\dot{\vph}\|_{L^2})^2 + \|\dot{\vph}\|^2_{L^2}. $$
\end{proof}

\begin{prop}
\label{prop;Hardy-ineq-2} For any $\vph \in \D$, 
$$ \int_{t_0}^{\infty} \frac{|\vph(t)-\vph(t_0)|^2}{w^2(t)} \,dt \le 4 \int_{t_0}^{\infty} |\vd(t)|^2 \,dt, \; \text{ if } t_0 >-1;$$
$$ \int_{-\infty}^{t_0} \frac{|\vph(t)-\vph(t_0)|^2}{w^2(t)} \,dt \le 4 \int_{-\infty}^{t_0} |\vd(t)|^2 \,dt, \; \text{ if } t_0 <1.$$
\end{prop}
\begin{proof}
We only prove the first inequality, as the second can be proven similarly. 

Assume $t_0>-1$ and for any $T>t_0$, let 
\[
 I_T=\int_{t_0}^{T}\frac{|\vph(t)-\vph(t_0)|^2}{(1+t)^2}\,dt.
\]
Using integration by parts and the Cauchy-Schwarz inequality
\[
 \begin{split}
 I_T
 &=-\frac{|\vph(T)-\vph(t_0)|^2}{1+T}
   +2\int_{t_0}^{T}\frac{\langle \vph(t)-\vph(t_0),\dot\vph(t)\rangle}{1+t}\,dt\\
 &\le 2I_T^{1/2}
       \left(\int_{t_0}^{T}|\dot\vph(t)|^2\,dt\right)^{1/2}.
 \end{split}
\]
As a result, $I_T\le4\int_{t_0}^{T}|\dot\vph|^2\,dt$. Notice that for any $t\ge t_0$,
$$w(t)=1+|t|\ge1+t>0.$$ 
Letting $T\to\infty$ and by the monotone convergence theorem
\[
 \int_{t_0}^{\infty}\frac{|\vph(t)-\vph(t_0)|^2}{w(t)^2}\,dt
 \le\int_{t_0}^{\infty}\frac{|\vph(t)-\vph(t_0)|^2}{(1+t)^2}\,dt
 \le4\int_{t_0}^{\infty}|\dot\vph(t)|^2\,dt.
\]
\end{proof}

By \eqref{eq;D-pointwise}, a weakly convergent sequence in $\D$ converges uniformly on each compact interval, as the sequence is bounded in $H^1$ norm. 

\begin{lem}\label{lem;normalized-closed}
$\Lmd^{a,b}_0$ is sequentially weakly closed in $\D$ and contains a path of finite renormalized action.
\end{lem}
\begin{proof}
Let $\vph_n\rightharpoonup\vph$ in $\D$ and
$\vph_n\in\Lmd^{a,b}_0$. Write $x_n=\xb+\vph_n$ and $x=\xb+\vph$.
After passing to a subsequence, the corresponding normalization times
satisfy $t_n\to t_0\in[-1/2,1/2]$. Local uniform convergence gives
$|x(t_0)|\le r_0$.

It remains to check the tied condition. By \eqref{eq;D-pointwise}, the
reference path dominates all the perturbations uniformly at infinity:
its growth is of order $|t|$ in the hyperbolic case and $|t|^{2/3}$
in the parabolic case. Consequently there is a common $T>1$ such that
\[
 |x_n(t)|>r_0+1,\qquad |x(t)|>r_0+1,\qquad |t|\ge T.
\]
Since $x_n|_{[-T,T]} \in \Gm^{a,b}_{-T, T}(x_n)$, for each $n$, the uniform convergence of $x_n|_{[-T,T]}$ to $x|_{[-T, T]}$ implies $x|_{[-T, T]} \in\Gm^{a,b}_{-T, T}(x) $. Hence $\vph \in \Lmd_0^{a, b}$.  

For nonemptiness, choose a path equal to $\xb$ outside a compact
interval, and near $t=0$ set
\[
 x(t)=q_a(t)+|t|^{2/3}e,
 \qquad e\in\mathbb S^1.
\]
Choose the interval small enough that this arc meets no other primary,
and join it to the reference tails by piecewise smooth collision-free
arcs. Such joining arcs exist because the primary trajectories are
disjoint and a plane with finitely many punctures is path connected.
Near zero both $|\dot x|^2$ and $U(x,t)$ are $O(|t|^{-2/3})$ and are
integrable. Thus $\vph=x-\xb\in\D$ has finite action, is normalized
at zero, and belongs to $\Lmd^{a,b}$ by Proposition \ref{prop;Gamma}.
\end{proof}

By the above result, we can show there is a global minimizer of $\Ar$ in $\Lmd_0^{a, b}$ as stated in the next Proposition, whose proof will be given in Section \ref{sec;hyp-scatter} and Section \ref{subsec;bi-para}. 

\begin{prop}
    \label{prop;Ar-minimizer} For any $1 \le a < b \le N$, $\Ar$ has a global minimizer $\vph(t)$ in $\Lmd_0^{a,b}$. 
\end{prop}
 
For the rest of the section let $\vph(t)$ be a minimizer of $\Ar$ in $\Lmd_0^{a, b}$ and let $x = \xb + \vph$. 
\begin{lem}
    \label{lem;NormalMinCollFree} For any $T >0$ large enough, $x|_{[-T, T]}$ is a minimizer of $\A$ in $\Gm^{a, b}_{-T, T}(x)$.
\end{lem}
\begin{proof} First we claim the following holds
$$ \Ar(\vph) = \inf_{\psi \in \Lmd_0^{a, b}} \Ar(\psi) = \inf_{\psi \in \Lmd^{a, b}} \Ar(\psi). $$
To see this, choose an arbitrary $\psi\in\Lmd^{a,b}$. For any $k\in\zz$, by \eqref{eq;Tk}
$$ \xb(t)+T_k\psi(t)=(\xb+\psi)(t+k).$$
By the definition of $\Lmd^{a, b}, \Lmd_0^{a, b}$, there exist a $k \in \zz$, such that $T_k \psi \in \Lmd_0^{a, b}$. Then Lemma \ref{lem;Ah-inv-timeshift} implies
$$ \Ar(\vph) \le \Ar(T_k \psi) = \Ar(\psi).
$$
By a contradiction argument, let's assume for some $T$ large enough, there is a  $y \in \Gm_{-T, T}^{a, b}(x)$ such that $\A(y; -T, T) < \A(x; -T, T)$. Then the following path $\psi$ belongs to $\Lmd^{a, b}$
    $$ \psi(t) = \begin{cases} \vph(t), &  \text{ if } t \notin [-T, T]; \\
    y(t) - \xb(t), & \text{ if } t \in [-T, T],     
    \end{cases}
    $$
    and by a direct computation
    $$  \Ar(\vph) -\Ar(\psi) = \A(x; -T, T) - \A(y; -T, T) > 0,$$
    which is a contradiction to our claim. 
\end{proof}

By the above lemma and Proposition \ref{prop;mini-tied}, $x$ is either a collision-free solution or a regularizable solution of \eqref{eq;RNBP1} with at most one collision with $m_a$ or $m_b$. We complete this section by showing $x$ satisfies the asymptotic expressions given in \eqref{eq;scatter}, and we will only give the details for $t\to+\infty$.  

First let's consider the hyperbolic case. Notice that \eqref{eq;xb-bi-hyp-1} and \eqref{eq;D-pointwise} imply 
$$ x(t)=\sqrt{2h} t\bs^++O(\sqrt t), \quad t \to +\infty.
$$ 
Since $\ddot{\xb}(t)=0$, when $t >1$, and $x(t)$ satisfies \eqref{eq;RNBP1}, we have
\[
 \ddot\vph(t)=\ddot x(t)
 =-\frac{m}{2h t^2}\bs^++O(t^{-5/2}).
\]
As this is integrable, $\dot\vph$ has a limit at infinity.
Since $\dot\vph\in L^2$, that limit is zero. Integrating from infinity
and then once more gives
\[
 \dot\vph(t)=\frac{m}{2h t}\bs^++O(t^{-3/2}),\qquad
 \vph(t)=\frac{m}{2h}(\log t)\bs^++O(1).
\]
This proves the estimate for the hyperbolic case in \eqref{eq;scatter}.

The estimate for the parabolic case in \eqref{eq;scatter} follows directly from \eqref{eq;xb-bi-para-1} and \eqref{eq;D-pointwise}, as
\[
 x(t)=\al t^{2/3}\bs^+ +O(t^{1/2})
      =\al t^{2/3}\bs^+ +o(t^{2/3}),
 \qquad t\to +\infty.
\]
We may also get an estimate of $\dot{x}(t)$. Again as $x$ is a solution of \eqref{eq;RNBP1}, by \eqref{eq;xb-qi-lbd-bp}, 
\[
 \ddot x(t)=-\frac{m}{\al^2}t^{-4/3}\bs^++O(t^{-3/2}).
\]
Hence $\dot x$ has a finite limit, and it must be zero, as
$x(t)=o(t)$. Then integrating the above equation from infinity gives
\[
 \dot x(t)=\frac{2\al}{3}t^{-1/3}\bs^++O(t^{-1/2}).
\]

%%%%%%%%%%%%%%%%%%%%%%%%%%%%%%%%%%%%%%%%%%%

\section{Proof of Proposition \ref{prop;Ar-minimizer}: hyperbolic case}
\label{sec;hyp-scatter}

Throughout this section the reference path $\xb$ is given as \eqref{eq;xb-bi-hyp-1}, and we set 
$$ c=\sqrt{2h}, \quad m=\sum_i m_i, \quad w(t)=1+|t|. $$
The existence of a global minimizer of $\Ar$ in $\Lmd_0^{a, b}$ will follow from a standard argument in the direct method of calculus of variations, once the coercivity and weak lower semicontinuity of $\Ar$ are established in the next two subsections. 

\subsection{Coercivity} First notice that 
\begin{equation}\label{eq;xb-qi-lbd}
 |\xb(t)-q_i(t)|\ge c\quad \forall t\in\rr,\qquad
 |\xb(t)-q_i(t)|\ge c w(t)\quad \forall |t|\ge1.
\end{equation}
For any $\vph\in\Lmd^{a,b}_0$, let $t_0$ be its normalization time. By \eqref{eq;vph-t0-upb}, \eqref{eq;D-pointwise} and the Cauchy-Schwarz inequality
$$|\vph(t)|\le r_1+\|\dot\vph\|_{L^2}\sqrt{w(t)}. $$
Since $\ddot\xb$ vanishes outside
$[-1,1]$,  
\begin{equation}\label{eq;bi-hyp-xdd}
 \int_\rr|\la\ddot\xb,\vph\ra|\,dt\le C(1+\|\dot\vph\|_{L^2}).
\end{equation}
Moreover, for $|t|\ge1$,
\[
 |\xb(t)+\vph(t)-q_i(t)|\le cw(t)+\|\dot\vph\|_{L^2}\sqrt{w(t)}+r_2,
 \qquad r_2=2r_0+r_1.
\]
It follows that
\[
 U(\xb,t)-U(\xb+\vph,t)
 \le \frac{m}{cw(t)}-
 \frac{m}{cw(t)+\|\dot\vph\|_{L^2}\sqrt{w(t)}+r_2}.
\]
To show the right side of the above inequality is integrable, set 
$$ w=s^2, \quad A=\|\dot\vph\|_{L^2}/c, \quad B=r_2/c. $$
Then \[ \int_{1}^{\infty} U(\xb,t)-U(\xb+\vph,t) \,dt \le 
 \frac{2m}{c}\int_{\sqrt2}^{\infty}
 \left(\frac1s-\frac{s}{s^2+As+B}\right)ds.
\]
Notice that when $s\ge\sqrt2$, 
 $$ s^2+As+B\le s(s+D), \quad \text{ where } D=A+B/\sqrt2$$ 
As a result, 
\[ \int_{1}^{\infty} U(\xb,t)-U(\xb+\vph,t) \,dt \le 
 \frac{2m}{c}\int_{\sqrt2}^{\infty}
 \left(\frac1s-\frac1{s+D}\right)ds
 =\frac{2m}{c}\log\left(1+\frac{D}{\sqrt2}\right).
\]
Since the same estimate holds on $t \le -1$, we get 
\begin{equation}\label{eq;bi-hyp-U}
 \int_{|t|\ge1}\bigl(U(\xb,t)-U(\xb+\vph,t)\bigr)\,dt
 \le \frac{4m}{c}\log(1+\|\dot\vph\|_{L^2})+C.
\end{equation}
On $[-1,1]$, positivity of $U$ and \eqref{eq;xb-qi-lbd} give an
upper bound $2m/c$ for the same difference. Combining these estimates,
\begin{equation}\label{eq;hyp-coercivity}
 \Ar(\vph)\ge\ey \|\dot\vph\|_{L^2}^2-C \|\dot\vph\|_{L^2}-\frac{4m}{c}\log(1+\|\dot\vph\|_{L^2})-C.
\end{equation}
By Lemma \ref{lem;Dnorm-L2norm}, this tends to infinity whenever
$\|\vph\|_{\D}\to\infty$ in $\Lmd^{a,b}_0$.  

\subsection{Weak lower semicontinuity}
Let $\vph_n\rightharpoonup\vph$ in $\D$. The sequence is bounded,
and \eqref{eq;D-pointwise} implies there is a $T>1$ such that, for every
$i$, $\theta\in[0,1]$ and $|t|\ge T$,
\[
 |\xb(t)-q_i(t)+\theta\vph_n(t)|\ge \frac c2 w(t).
\]
The same bound holds for $\vph$, by uniform convergence. Applying the mean value formula to
$z\mapsto |z|^{-1}$ yields
\begin{equation}\label{eq;hyp-tail-domination}
 |U(\xb+\vph_n,t)-U(\xb,t)|
 \le C\frac{|\vph_n(t)|}{w(t)^2}\le Cw(t)^{-3/2}
\end{equation}
which is absolutely integrable on $\{|t| \ge T\}$. The dominated convergence theorem then implies
\[
 \lim_{n\to\infty}\int_{|t|>T}
 \bigl(U(\xb+\vph_n,t)-U(\xb,t)\bigr)\,dt
 =\int_{|t|>T}\bigl(U(\xb+\vph,t)-U(\xb,t)\bigr)\,dt.
\]
On the remaining interval, Fatou's lemma gives
\[
 \int_{-T}^{T}U(\xb+\vph,t)\,dt
 \le\liminf_{n\to\infty}\int_{-T}^{T}U(\xb+\vph_n,t)\,dt.
\]
Finally, the linear term involving $\ddot\xb$ converges by local
uniform convergence, and the kinetic integral is weakly lower
semicontinuous. Hence
\begin{equation}\label{eq;hyp-lsc}
 \Ar(\vph)\le\liminf_{n\to\infty}\Ar(\vph_n).
\end{equation}

\section{Proof of Proposition \ref{prop;Ar-minimizer}: parabolic case}
\label{subsec;bi-para}

In this section the reference path $\xb$ is given as \eqref{eq;xb-bi-para-1}. Then for any $1 \le i \le N$, 
\begin{equation}\label{eq;xb-qi-lbd-bp}
 |\xb(t)-q_i(t)|\ge\al\quad(t\in\rr),\qquad
 |\xb(t)-q_i(t)|\ge\al w(t)^{2/3}\quad(|t|\ge1).
\end{equation}
where $\al =(9m/2)^{1/3}$ and $w(t)=1+|t|$. Now set
\begin{equation}
\label{eq;F-H} F(t)=\nabla U(\xb(t),t)-\ddot\xb(t) \; \text{ and } \; H(t)=\nabla^2U(\xb(t),t) 
\end{equation}
\begin{lem}\label{lem;parabolic-residual}
There is a constant $C>0$ such that for every $z\in\rr^2$ and $|t|\ge1$,
\begin{equation}\label{eq;nU-xdd-ubd}
 |F(t)|\le Cw(t)^{-2}, \quad -\frac{2|z|^2}{9w(t)^2}
 \le\la H(t)z,z\ra\le\frac{4|z|^2}{9w(t)^2}.
\end{equation} 
\end{lem}
\begin{proof}
We will only give the details for $t\ge1$, set
\[
 \qt_i(t)=\frac{q_i(t)-r_0\bs^+}{\al(t+1)^{2/3}},
 \qquad |\qt_i(t)|\le\frac{2r_0}{\al}(t+1)^{-2/3}.
\]
As the map $G(z)=z/|z|^3$ has bounded derivative in a neighborhood of
$\bs^+$, for $t$ large enough
$$ G(\bs^+-\qt_i)=\bs^++O((t+1)^{-2/3}). $$
Since $m/\al^2=2\al/9$, this implies
\[
 \nabla U(\xb,t)
 =-\frac{1}{\al^2(t+1)^{4/3}}
       \sum_i m_iG(\bs^+-\qt_i)
 =-\frac{2\al}{9}(t+1)^{-4/3}\bs^++O((t+1)^{-2}).
\]
Combining this with $\ddot\xb=-(2\al/9)(t+1)^{-4/3}\bs^+$ proves the first estimate in \eqref{eq;nU-xdd-ubd}.

For the rest, by a direct differentiation 
\[
 H(t)=\sum_i m_i\left(
 \frac{3(\xb-q_i)(\xb-q_i)^{\mathsf T}}{|\xb-q_i|^5}
 -\frac{I}{|\xb-q_i|^3}\right).
\]
Notice that for each $i$,  $2/|\xb-q_i|^3$ and $-1/|\xb-q_i|^3$ are the eigenvalues of  
$$ \frac{3(\xb-q_i)(\xb-q_i)^{\mathsf T}}{|\xb-q_i|^5}
 -\frac{I}{|\xb-q_i|^3}.$$
The second estimate in \eqref{eq;nU-xdd-ubd} then follows from  \eqref{eq;xb-qi-lbd-bp}  as $m/\al^3=2/9$.
\end{proof}
With the above lemma we can prove the coercivity and weak lower semicontinuity of $\Ar$ in $\Lmd_0^{a, b}$ in the next two subsections, and then the existence of a global minimizer of $\Ar$ in $\Lmd_0^{a, b}$ follows from a standard argument in the direct methods of calculus of variations. 

\subsection{Coercivity}
For a collision-free point $x=\xb+\vph$, put
\[
 u_i=|\xb-q_i|^2,\qquad
 v_i=|x-q_i|^2-u_i=2\la\xb-q_i,\vph\ra+|\vph|^2.
\]
Taylor's formula applied to $u\mapsto u^{-1/2}$ gives
\begin{equation}\label{eq;positive-remainder}
  \frac{1}{|x -q_i|}- \frac{1}{|\xb -q_i|}  = (u_i+v_i)^{-1/2}-u_i^{-1/2}
 =-\ey u_i^{-3/2}v_i +R(u_i, v_i)
\end{equation}
where 
$$  R(u_i, v_i)=\frac34v_i^2\int_0^1(1-\theta)
                  (u_i+\theta v_i)^{-5/2}\,d\theta\ge0$$
as for all $0 \le \tht \le 1$,
$$ u_i + \tht v_i = (1 -\tht)|\xb -q_i|^2  + \tht |x -q_i|^2  \ge 0. $$
Notic that this above inequality is valid even if the spatial segment from $\xb$ to $x$ passes through a primary, because $u_i+\theta v_i$ is a convex combination of the two positive squared endpoint distances.

As a result, for $F$ given as in \eqref{eq;F-H}, we have
\begin{equation}\label{eq;Ux-Uxb-bp}
 U(x,t)-U(\xb,t)-\la\ddot\xb,\vph\ra
 \ge-\ey\sum_i\frac{m_i|\vph|^2}{|\xb-q_i|^3}
       +\la F,\vph\ra.
\end{equation}
Without loss of generality, we may assume $t_0\in[-1/2,1/2]$ satisfies
\eqref{eq;vph-t0-upb}. Then 
$$ |\vph(t_0)|\le r_1=\al+2r_0 \; \text{ and } |\vph(t) - \vph(t_0)| \le\|\dot\vph\|_{L^2} \sqrt{|t-t_0|}. 
$$ 
By Proposition \ref{prop;Hardy-ineq-2} and \eqref{eq;xb-qi-lbd-bp}
$$ \ey\int_{|t|\ge1}\sum_i
       \frac{m_i|\vph(t_0)|^2}{|\xb-q_i|^3}\,dt\le C;$$
$$\int_{|t|\ge1}\sum_i
       \frac{m_i|\la \vph -\vph(t_0), \vph(t_0)\ra|}{|\xb-q_i|^3}\,dt
 \le C_1 \|\dot{\vph}\|_{L^2}\int_{|t|\ge1}w(t)^{-3/2}\,dt\le C_2 \|\dot{\vph}\|_{L^2};$$
 \begin{equation*}
 \begin{split}
 \ey\int_{|t|\ge1}\sum_i
       \frac{m_i|\vph(t) - \vph(t_0)|^2}{|\xb(t)-q_i(t)|^3}\,dt
 &\le\frac19\left(
 \int_1^\infty\frac{|\vph(t) - \vph(t_0)|^2}{(t+1)^2}\,dt
 +\int_{-\infty}^{-1}\frac{|\vph(t) - \vph(t_0)|^2}{(1-t)^2}\,dt\right)\\
 &\le\frac49\left(
 \int_{t_0}^\infty|\dot\vph|^2\,dt
 +\int_{-\infty}^{t_0}|\dot\vph|^2\,dt\right)
 =\frac49\|\dot{\vph}\|_{L^2}^2.
 \end{split}
 \end{equation*} 
Together they imply
\begin{equation} \label{eq;BP3}
\ey \int_{|t| \ge 1} \sum_i\frac{m_i|\vph|^2}{|\xb-q_i|^3} \,dt \le \frac{4}{9}
\|\dot{\vph}\|_{L^2}^2 + C_1 \|\dot{\vph}\|_{L^2} + C_2. 
\end{equation}
At the same time by Lemma \ref{lem;parabolic-residual},
\begin{equation}\label{eq;BP4}
 \int_{|t|\ge1}|\la F,\vph\ra|\,dt
 \le C\int_{|t|\ge1}\frac{r_1+\|\dot{\vph}\|_{L^2}\sqrt{w(t)}}{w(t)^2}\,dt
 \le C(1+\|\dot{\vph}\|_{L^2}).
\end{equation}
Meanwhile on $[-1,1]$, as $U(x,t)$ is always positive and $\xb$  stays away from all primaries,
\begin{equation}
\label{eq;BP5} \int_{-1}^{1} U(x,t)-U(\xb,t)-\la\ddot\xb,\vph\ra \,dt
 \ge-C(1+\|\dot{\vph}\|_{L^2}).
\end{equation}
Combining \eqref{eq;Ux-Uxb-bp}--\eqref{eq;BP5}, we conclude that
\begin{equation}\label{eq;parabolic-coercivity}
 \Ar(\vph)\ge \frac1{18}\|\dot{\vph}\|_{L^2}^2-C_1\|\dot{\vph}\|_{L^2}-C_2.
\end{equation}
By Lemma \ref{lem;Dnorm-L2norm}, $\Ar$ is coercive on
$\Lmd^{a,b}_0$.

\subsection{Weak lower semicontinuity}
Let $\vph_n\rightharpoonup\vph$ in $\D$. By \eqref{eq;D-pointwise} there is a positive constant $C$ independent of $n$ such that
$$  |\vph_n(t)|, |\vph(t)| \le C \sqrt{w(t)}. $$
Since $w^{1/2}(t)=o(w^{2/3}(t))$, we can fix a $T\ge1$ large enough, such that 
\begin{equation}\label{eq;para-segment-distance}
 |\xb(t)-q_i(t)+\theta\vph_n(t)|
 \ge\frac\al2 w(t)^{2/3},
 \qquad |t| \ge T,\quad 0\le\theta\le1.
\end{equation}
By Taylor's formula, when $|t| \ge T$
\begin{equation}
\label{eq;U-diff} U(\xb+\vph_n,t)-U(\xb,t)-\la\ddot\xb,\vph_n\ra
 =\la F,\vph_n\ra+\ey\la H\vph_n,\vph_n\ra
       +\mathcal R(t,\vph_n)
\end{equation}
where the remainder is given by
\begin{equation}\label{eq;R-bp}
 \mathcal R(t,z)=\ey\int_0^1(1-\theta)^2
 D_x^3U(\xb(t)+\theta z,t)[z,z,z]\,d\theta.
\end{equation}
Notice that $|\la F(t),\vph_n(t)\ra|\le Cw^{-3/2}(t)$ and by \eqref{eq;para-segment-distance}, 
\begin{equation}\label{eq;para-remainder-domination}
 |\mathcal R(t,\vph_n(t))|
 \le Cw(t)^{-8/3}|\vph_n(t)|^3
 \le Cw(t)^{-7/6}, \;\;  \text{ when } |t| \ge T. 
\end{equation}
Hence $|\la F,\vph_n\ra|+ |\mathcal R(t,\vph_n(t))|$ is integrable on $|t| \ge T$, and by the dominated convergence theorem
\begin{equation}
\label{eq;lim-F-R} \int_{|t| \ge T} \la F,\vph\ra + \mathcal R(t,\vph) \,dt =
 \lim_{n \to \infty} \int_{|t| \ge T} \la F,\vph_n\ra + \mathcal R(t,\vph_n) \,dt < \infty. 
\end{equation}
For $\psi = \vph$ or $\vph_n$ with $T$ chosen as above,  we have the renormalized action functional as below
\begin{equation}\label{eq;para-lsc-decomposition}
 \begin{split}
 \Ar(\psi)={}&Q_T(\psi)+J_T(\psi)
 +\int_{|t| \ge T}\bigl(\la F,\psi\ra+\mathcal R(t,\psi)\bigr)\,dt\\
 &+\int_{-T}^{T}\bigl(U(\xb+\psi,t)-U(\xb,t)
                         -\la\ddot\xb,\psi\ra\bigr)\,dtb
 \end{split}
\end{equation}
where $Q_T(\psi)$ and $J_T(\psi)$ are defined as
\begin{equation}\label{eq;global-quadratic-form}
 Q_T(\psi)=\ey\int_\rr|\dot\psi|^2\,dt
 +\ey\int_{|t| \ge T}\la H(t)(\psi(t)-\psi(0)),
                         \psi(t)-\psi(0)\ra\,dt;
\end{equation}
\begin{equation}
\label{eq;J-T}
 J_T(\psi)=\int_{|t| \ge T}\left(
 \la H(t)(\psi(t)-\psi(0)),\psi(0)\ra
 +\ey\la H(t)\psi(0),\psi(0)\ra\right)\,dt
\end{equation}

By Hardy's inequality with base point zero,  
\[
 \int_\rr\frac{|\psi(t)-\psi(0)|^2}{w(t)^2}\,dt
 \le4\int_\rr|\dot\psi|^2\,dt.
\]
Combining this with the second estimate in \eqref{eq;nU-xdd-ubd} implies
\begin{equation}\label{eq;global-quadratic-bounds}
 \frac1{18}\|\dot\psi\|_{L^2}^2
 \le Q_T(\psi)
 \le\frac{25}{18}\|\dot\psi\|_{L^2}^2.
\end{equation}
This means $Q_T$ is a continuous nonnegative quadratic form on
$\D$. Hence it is convex and weakly lower semicontinuous. 

Meanwhile as $|\vph_n(t)-\vph_n(0)|\le C\sqrt{w(t)}$ and $\|H(t)\|\le Cw(t)^{-2}$
$$
\begin{aligned}
    |J_T(\vph_n)| & \le \int_{|t| \ge T} \left|
 \la H(t)(\vph_n(t)-\vph_n(0)),\vph_n(0)\ra
 +\ey\la H(t)\vph_n(0),\vph_n(0)\ra\right|\,dt \\
 & \le\int_{|t| \ge T} C_1 \left( w^{-3/2}(t) + w^{-2}(t) \right) \,dt \le C_2. 
\end{aligned}
$$
Again the dominated convergence theorem then implies $J_T(\vph_n)\to J_T(\vph)$.

Meanwhile on $[-T,T]$, as the potential function is always positive, by Fatou's lemma
$$ \int_{-T}^T U(\xb + \vph, t)\,dt \le \liminf_{n \to \infty} \int_{-T}^T U(\xb + \vph_n, t)\,dt$$
and by the uniform convergence of $\vph_n$, 
$$  \int_{-T}^T U(\xb , t) + \la \ddot{\xb}, \vph \ra\,dt = \lim_{n \to \infty}   \int_{-T}^T U(\xb , t) + \la \ddot{\xb}, \vph_n \ra\,dt. $$
This proves the weak lower semicontinuity of $\Ar$, as all the above estimates imply 
\begin{equation}\label{eq;para-lsc}
 \Ar(\vph)\le\liminf_{n\to\infty}\Ar(\vph_n).
\end{equation}

\section{Proof of Lemma \ref{lem;reg}} \label{sec;reg}

Before proving Lemma \ref{lem;reg}, we recall a well-known result for the planar Kepler problem
$$\ddot{z} = -\mu \frac{z}{|z|^3}, \quad z \in \rr^2 \text{ with mass } \mu >0 .$$
Given an arbitrary pair $\bs^{\pm} \in \mathbb{S}^1$, there is a corresponding parabolic homothetic collision-ejection solution $\bar{z}: \rr \to \rr^2$ of the Kepler problem
$$
\bar{z}(t) = \begin{cases}
\left(\frac{9}{2}\mu \right)^{\frac{1}{3}} |t|^{\frac{2}{3}} \bs^-, & \text{ if } t \le 0; \\
\left(\frac{9}{2}\mu \right)^{\frac{1}{3}} |t|^{\frac{2}{3}} \bs^+, & \text{ if } t \ge 0.
\end{cases}
$$
The following result is well-known and different proofs can be found in \cite{TV07}, \cite{FGN11}, \cite{Y15a} and \cite{Chen17}. 
\begin{prop}
\label{prop;Kepler} When $\bs^- \ne \bs^+$, for any $T>0$, there are two different collision-free solutions $z^{\pm}: [-T, T] \to \rr^2$ of the Kepler problem, such that both go from $\bar{z}(-T)$ to $\bar{z}(T)$ around the origin, but with one clockwise and the other counter-clockwise, and both of their action values are smaller than the action value of $\bar{z}$
$$ \int_{-T}^T \ey |\dot{z}^{\pm}|^2 + \frac{\mu}{|z^{\pm}|} \,dt < \int_{-T}^T \ey |\dot{\bar{z}}|^2 + \frac{\mu}{|\bar{z}|} \,dt. $$ 
\end{prop}

Now let's start our proof of Lemma \ref{lem;reg}. Following the notation and results from Section \ref{sec;mini-top-const}, $x$ is a minimizer of $\A$ in $\Gm_{t_1, t_2}^{a, b}(p, q)$ with the set of collision moments $Z_x = \{t_0\}$ with $x(t_0) = q_{i_0}(t_0)$, $i_0 =a$ or $b$. Without loss of generality, let's assume $i_0=1$ below and define the relative path as
$$ \br(t) = x(t) - q_1(t), \quad t \in [t_1, t_2].  $$
The action $\A(x; t_1, t_2)$ then can be written as
$$ \A(\br; t_1, t_2) := \int_{t_1}^{t_2} \frac{1}{2}|\dot{\br} + \dot{q}_{1}|^2  +  U(\br, t) \,dt. $$
where for any $z \in \rr^2$ and $t \in \rr$,
$$ U(z, t) = \frac{m_1}{|z|} + W(z, t), \quad W(z, t) = \sum_{i \ne 1} \frac{m_i}{|z + q_1(t) - q_i(t)|}. $$

By property (a) in Proposition \ref{prop;Sundman}, there exists a pair of asymptotic directions $\bs^{\pm} \in \mathbb{S}^1$ satisfying
$$ \bs^{\pm} = \lim_{t \to t_0^{\pm}} \br(t)/|\br(t)|. $$
First we will show $\bs^- = \bs^+$. Choose $\dl>0$ small enough and set $I=[t_0-\dl,t_0+\dl]\subset(t_1,t_2)$ below. When $\bs^- \ne \bs^+$, we can find two small fixed-end local deformations of $\br$ on $I$ to obtain two new collision-free paths $\br^{\pm}$ with one going around the origin clockwise and the other counter-clockwise, and  satisfying
$$\A(\br^{\pm}; t_1, t_2) < \A(\br; t_1, t_2).$$ 
By the definition of tying and $\Gm_{t_1, t_2}^{a, b}(p, q)$, one of the paths $x^{\pm} = \br^{\pm} + q_1$ must be contained in $\Gm_{t_1, t_2}^{a, b}(p, q)$, which is absurd as 
$$ \A(x^{\pm}; t_1, t_2) = \A(\br^{\pm}; t_1, t_2 ) < \A(\br; t_1, t_2) = \A(x; t_1, t_2). $$
The existence of such two local deformations $\br^{\pm}|_{I}$ is based on Proposition \ref{prop;Kepler} and a blow-up argument. The details can be found in \cite{CY26}. 

Now for any $z \in \rr^2$ and $t \in \rr$, define
 $$ V(z,t)=U(z,t)-\la\ddot q_1(t), z\ra $$
and since $\br(t) =0$ only when $t=t_0$, the energy $E(t)$ is well-defined for all $t \ne t_0$
$$ E(t)=|\dot{\br}(t)|^2/2-V(\br(t),t). $$
By property (b) in Proposition \ref{prop;Sundman}, to show $x$ is regularizable, it is enough to prove  
\begin{equation}\label{eq;relative-energy-matching}
 \lim_{t \to t_0^+} E(t) = \lim_{t \to t_0^-} E(t). 
\end{equation}

Since $q(t)$ is always collision-free, for $\dl$ small enough there
is a neighborhood of $\{(\br(t),t):t\in I\}$, including $(0,t_0)$,
on which $W$ is smooth. Define
\[ 
 \mathcal J(\br;I)=\int_I\left(\ey|\dot \br|^2+V(\br,t)\right)\,dt.
\]
Integration by parts in the mixed kinetic term gives
\begin{equation}
\label{eq;J} \A(\br;I)=\mathcal J(\br;I)
 +[\la \br,\dot q_1\ra]_{t_0-\dl}^{t_0+\dl}
 +\ey\int_I|\dot q_1|^2\,dt.
\end{equation}
Notice that the last two terms are unchanged by variations with fixed endpoints.

Meanwhile although $V$ is singular at $\br=0$, its partial
derivative with respect to time is smooth near $(0,t_0)$ when $\br$ is kept fixed
\begin{equation} \label{eq;V_t}
 \partial_tV(\br,t)=\partial_tW(\br,t)-\la \frac{d^3}{dt^3}q_1(t),\br\ra
\end{equation}
 In particular, $\partial_tV(\br(t),t)$ is bounded on $I$. 

For any $\eta\in C_c^\infty((t_0-\dl,t_0+\dl))$ and $|\ep| <<1$ small enough, set 
$$  \theta_\ep: I \to I; \quad s \to s+\ep\eta(s). $$
Notice that each $\theta_\ep$ is an increasing diffeomorphism of $I$, equal to the identity near its endpoints. Then we can define new relative paths as 
\[
  \br_\ep(t)= \begin{cases}
      \br(\theta^{-1}_\ep(t)),\qquad & t \in I; \\
      \br(t), \qquad & t \in [t_1, t_2] \setminus I.
  \end{cases}
\]
Using each $\br_{\ep}$, we can define a new $H^1$ path 
\[
 x_\ep(t)=q_1(t)+\br_\ep(t), \quad t \in [t_1, t_2]
\]
which has the same endpoints as $x$ and still collides
with the same primary, as when $t_{\ep} = \tht_{\ep}(t_0)$
$$ \br_{\ep}(t_{\ep}) = \br(\tht^{-1}_{\ep}(t_{\ep})) = \br(t_0)=0.$$
This means $x_{\ep}(t_{\ep}) = q_1(t_{\ep})$, and by Proposition \ref{prop;Gamma}, $x_\ep\in\Gm_{t_1,t_2}^{a,b}(p,q)$. As a result
 $$ \A(x; t_1, t_2) \le \A(x_{\ep}; t_1, t_2)$$
 and then by \eqref{eq;J}
$$ \mathcal J(\br;I)\le\mathcal J(\br_\ep;I). $$

Changing the time parameter from $t$ to $s$ in $\mathcal J$ gives
\[
 \mathcal J(\br_\ep;I)
 =\int_I\left\{
 \frac{|\dot{\br}(s)|^2}{2(1+\ep\eta'(s))}
 +V(\br(s),s+\ep\eta(s))(1+\ep\eta'(s))
 \right\}\,ds.
\]
This is integrable, as the Sundman estimates in Proposition \ref{prop;Sundman} imply that near $t_0$
$$ |\dot{\br}(s)|^2=O(|s-t_0|^{-2/3}), \quad |\br(s)|^{-1} = O(|s-t_0|^{-2/3})
$$
Since the singular term $m_1/|\br|$ in $V$ is independent of time and the remaining time derivatives are
bounded on the neighborhood chosen above, we can differentiate $\mathcal J(\br_\ep;I)$ at $\ep=0$ and obtain
$$ \left.\frac{d}{d\ep}\mathcal J(\br_\ep;I)\right|_{\ep=0} =\int_I\left\{
 \left(-\ey|\dot{\br}|^2+V(\br,s)\right)\eta'
 +\partial_tV(\br,s)\eta\right\}\,ds =0. 
 $$ 
These imply $E\in L^1(I)$ and $E'=-\partial_tV(\br(t),t)$ in the sense of distributions on $I$. Then \eqref{eq;relative-energy-matching} follows from the fact that $\partial_tV(\br(t),t)$ is continuous near $t_0$.

%%%%%%%%%%%%%%%%%%%%%%%%%%%%%%%%%%%%%%%%%%

\hfill\newline
%\noindent{\bf Acknowledgement.}  
\bibliographystyle{abbrv}
\bibliography{References}

\end{document}